\documentclass{amsart}
\usepackage{graphicx}
\usepackage{amsmath}
\usepackage{amsthm}
\usepackage{amsfonts}

\usepackage{color}

\usepackage[all]{xy}

\newtheorem{theorem}{Theorem}[section]

\newtheorem{lemma}[theorem]{Lemma}
\newtheorem{proposition}[theorem]{Proposition}

\theoremstyle{definition}
\newtheorem{definition}[theorem]{Definition}
\theoremstyle{remark}
\newtheorem{remark}[theorem]{Remark}
\numberwithin{equation}{section}

\newcommand{\g}{\mathfrak{g}}

\newcommand{\ba}{\overline{\alpha}}

\begin{document}

\title[  ]
{On $n$-Hom-Leibniz algebras and cohomology}%
\author{Abdenacer  Makhlouf and Anita Naolekar }%
\address{A.M.  Universit\'{e} de Haute Alsace,  IRIMAS - D\'epartement de  Math\'{e}matiques,
6  bis, rue des Fr\`{e}res Lumi\`{e}re F-68093 Mulhouse, France.}

\address{A.N. Indian Statistical Institute, 8th Mile, Mysore Road,
Bangalore, 560059,
India.}

 \email{Abdenacer.Makhlouf@uha.fr, anita@isibang.ac.in}

\thanks {}

 \subjclass[2000]{}
\date{}
%
\begin{abstract}
The purpose of this paper is to provide a cohomology of $n$-Hom-Leibniz algebras. Moreover, we study some higher operations on cohomology spaces and deformations.

\end{abstract}
\maketitle

\section*{Introduction}
The first appearance of ternary operation goes back to 19th century,
where Cayley considered cubic matrices. The first motivation to study $n$-ary algebras appeared in Physics when Nambu suggested in 1973 a generalization of Hamiltonian Mechanics with  more than one hamiltonian \cite{Nambu}. The algebraic formulation of Nambu Mechanics is due to Takhtajan \cite{Takhtajan} and the abstract definition of $n$-Lie algebra is due to Filippov in 1985, see \cite{Filippov}. Moreover, 3-Lie algebras appeared in String Theory. Fuzzy sphere
(noncommutative space) arises naturally in the description of D1-branes ending on D3-
branes in Type IIB superstring theory and the effective dynamics of this system is described
by the Nahm equations. In \cite{Basu}, Basu and Harvey suggested to replace the Lie algebra appearing in the
Nahm equation by a 3-Lie algebra for the lifted Nahm equations. Furthermore, in the context
of Bagger-Lambert-Gustavsson model of multiple M2-branes, Bagger-Lambert managed
to construct, using a ternary bracket, an $N = 2$ supersymmetric version of the worldvolume
theory of the M-theory membrane, see \cite{BL}. Various algebraic aspects
of 3-Lie algebras, or more generally, of $n$-Lie algebras were considered. 
Construction, realization and classifications of 3-Lie algebras and $n$-Lie algebras were studied, see for example \cite{BSZ,De Azcarraga3,LarssonTA}. In particular,
representation theory of $n$-Lie algebras was first introduced by Kasymov in \cite{Kasymov1}. 
Through fundamental objects one may also represent a 3-Lie algebra and more generally
an $n$-Lie algebra by a Leibniz algebra \cite{Daletskii,Takhtajan1}. For deformation theory and cohomologies of $n$-Lie algebras, we refer to \cite{Gautheron1,Gautheron,R,Takhtajan1}.
 The $n$-Leibniz  algebras, which are a non-skewsymmetric version of $n$-Lie algebras, were introduced and studied in \cite{CasasLodayPirashvili}. Some of their properties were studied in \cite{CILL,De Azcarraga4}.
 Hom-type algebras were motivated by $q$-deformations of algebras of vector fields, in particular Witt and Virasoro algebras. It turns out that the structure obtained when usual derivation is replaced by a $\sigma$-derivation is no longer a Lie algebra but a twisted Jacobi condition is satisfied. These type of algebras were called Hom-Lie algebras and studied in \cite{HLS}. Generalizations to $n$-ary operations and $n$-Hom-Lie algebras were considered in \cite{Ataguema}. An extensive literature deals with properties of this kind of algebras, see \cite{AEM,ams:MabroukTernary,ams:MabroukRep,CS,MS}. 

In this paper, we consider $n$-Hom-Leibniz algebras and provide a cohomology complex with values in a representation. Moreover we study their one parameter formal deformations. In the first section, we define $n$-Hom-Leibniz algebras and endow a Hom-Leibniz algebra structure on the vector space $\mathcal{D}_{n-1}(L):= L^{\otimes n-1}$. In the second section, we define  representations of a $n$-Hom-Leibniz algebra and then a cohomology  with coefficients in a representation. In Section 3, we give an interpretation of the first cohomology groups  in terms of derivations and extensions. In Section 4, we show a relationship between the cohomology of $n$-Hom-Leibniz algebra and  its associated Hom-Leibniz algebra. Section 5 is dedicated to some higher operations and to show that the cohomology space of a $n$-Hom-Leibniz algebra turns into a graded Lie algebra. In the last section, we discuss formal  deformations and their connection to cohomology.

\section{$n$-Hom-Leibniz algebras}
In this paper all vector spaces are considered over a field $\mathbb{K}$ of characteristic 0. 
\begin{definition}
A (right) Hom-Leibniz algebra, \cite{MS},  is a triple $(L, [\cdot,\cdot ],\alpha)$ where $L$ is a vector space, with a bracket $[\cdot,\cdot ]: L\otimes L\longrightarrow L$ and a linear map $\alpha: L\longrightarrow L$ satisfying
\begin{equation}\label{homleib}
[[x, y],\alpha(z)]= [[x, z], \alpha(y)] + [\alpha(x), [y,z]].
\end{equation}
\end{definition}
Rewriting the above condition, if $$\begin{array}{rl}
ad_y: L&\longrightarrow L\\
x &\longrightarrow [x, y]
\end{array},
$$ then $$ad_{\alpha(z)}[x, y]= [ad_z(x), \alpha(y)] + [\alpha(x), ad_z(y)].$$

We introduce now the definition of $n$-Hom-Leibniz algebra. 
\begin{definition}
A (right) $n$-Hom-Leibniz algebra is a vector space $L$ together with a bracket $[\cdot,\cdots,\cdot]: L^{\otimes n}\longrightarrow L$ and $n-1$ linear maps 
$\alpha_i: L\longrightarrow L$, $1\leq i\leq n-1$, satisfying \\
$[[x_1,\cdots, x_n],\alpha_1(y_1), \cdots, \alpha_{n-1}(y_{n-1})]$
\begin{equation}\label{nhomleib}
=\sum_{i=1}^n [\alpha_1(x_1), \cdots, \alpha_{i-1}(x_{i-1}), [x_i, y_1, \cdots, y_{n-1}], \alpha_i(x_{i+1}), \cdots \alpha_{n-1}(x_n)].
\end{equation}
\end{definition}

Rewriting the above condition, for $Y=(y_1, \ldots, y_{n-1}) \in L^{\otimes n-1}$, if $$\begin{array}{rl}
ad_Y: L&\longrightarrow L\\
x &\longrightarrow [x, Y]
\end{array}
$$ then 

$$\begin{array}{rl}
& ad_{\alpha_1(y_1), \cdots, \alpha_{n-1}(y_{n-1})}[x_1, \cdots, x_{n-1}, x_n]\\
&= \sum_{i=1}^n [\alpha_1(x_1), \cdots, \alpha_{i-1}(x_{i-1}),ad_Y(x_i), \alpha_i (x_{i+1}), \cdots, \alpha_{n-1}(x_n)].
\end{array}$$

For $n=2$, this gives us a Hom-Leibniz algebra.
If we take $\alpha_i= id$ for all $1\leq i\leq n-1$ we get Leibniz $n$-algebras, \cite{CasasLodayPirashvili}.


\begin{definition}
Let $(A, [\cdot,\cdots,\cdot]_A, \alpha_1, \ldots \alpha_{n-1})$, $(B, [\cdot,\cdots,\cdot]_B, \beta_1, \ldots, \beta_{n-1})$ be two $n$-Hom-Leibniz algebras. A morphism $f: A\longrightarrow B$ is a linear map satisfying, for all $x_1,\cdots x_n\in A$, 
$$f([x_1, \cdots, x_n]_A)= [f(x_1), \cdots, f(x_n)]_B$$ and $f\circ \alpha_i= \beta_i\circ f$ for all $1\leq i\leq n-1.$
\end{definition}

\begin{remark}
A $n$-Hom-Leibniz algebra $(L, [\cdot,\cdots,\cdot], \alpha_1, .... \alpha_{n-1})$ will be simply denoted by $(L, [\cdot,\cdots,\cdot], \alpha)$ if $\alpha_1= \alpha_2= \cdots= \alpha_{n-1}=\alpha$.
\end{remark}

\begin{definition}
A $n$-Hom-Leibniz algebra $(L, [\cdot,\cdots,\cdot], \alpha)$ is said to be multiplicative if 
$$\alpha ([x_1, x_2, \cdots, x_n])= [\alpha (x_1), \alpha (x_2), \cdots, \alpha (x_n)].$$
\end{definition}

The following proposition allows to construct $n$-Hom-Leibniz algebra from a $n$-Leibniz algebra along a morphism.
\begin{proposition}\label{hom}
Let $A, [\cdot,...,\cdot])$ be a $n$-Leibniz algebra and $\alpha$ be an algebra morphism of $A$, i.e. $\alpha[x_1, \cdots, x_n]= [\alpha(x_1), \cdots,  \alpha(x_n)].$ Then $(A, \alpha[\cdot, \cdots, \cdot], \alpha)$ is a $n$-Hom-Leibniz algebra.
\end{proposition}
\begin{proof}The proof is straight forward.
\end{proof}
\begin{proposition}
If $(L, [\cdot,\cdots,\cdot], \alpha)$ is a $n$-Hom-Leibniz algebra, then $\mathcal{D}_{n-1}(L)= (L^{\otimes n-1}, \ba )$ is a Hom-Leibniz algebra with respect to the bracket
$$[a_1, \cdots, a_{n-1}, b_1, \cdots, b_{n-1}]= \sum_{i=1}^{n-1} \alpha (a_1) \otimes \cdots \otimes [a_i, b_1, \cdots, b_{n-1}]\otimes \cdots \otimes \alpha (a_{n-1})$$
and linear map $\ba : L^{\otimes n-1}\longrightarrow L^{\otimes n-1}$ defined by
$$\ba  (a_1\otimes \cdots \otimes a_{n-1})= (\alpha (a_1)\otimes \cdots\otimes \alpha (a_{n-1})).$$
Also, $\ba [X, Y]= [\ba (X), \ba (Y)]$ for all $X=(x_1, \cdots, x_{n-1})$ and $Y= (y_1, \cdots, y_{n-1})$ in $D_{n-1}(L)$.
\end{proposition}
{\bf Proof.} We shall check that $$[[X, Y], \alpha(Z)]= [[X, Z], \alpha(Y)] + [\alpha(X), [Y, Z]]. $$
Now, 
\begin{equation}\label{first}
\begin{array}{ll}
[[X, Y], \alpha(Z)]&=[\sum_{i=1}^{n-1} \alpha(x_1) \otimes \ldots \otimes \alpha(x_{i-1}) \otimes [x_i, y_1, \ldots, y_{n-1}]\otimes \ldots \otimes \alpha(x_{n-1}),\alpha(Z)]\\
&=\sum_{j=1}^{i-1}\sum_{i=1}^{n-1} \alpha^2(x_1) \otimes \ldots \otimes \alpha^2(x_{j-1})\otimes [\alpha(x_j), \alpha(Z)]\otimes \alpha^2 x_{j+1}\otimes \ldots \otimes \\
&\hspace{2cm}\alpha[x_i, y_1, \ldots, y_{n-1} ] \otimes \ldots \otimes \alpha^2(x_{n-1})\\
&+ \sum_{j=i+1}^{n-1}\sum_{i=1}^{n-1}\alpha^2(x_1) \otimes \ldots \otimes \alpha^2(x_{i-1})\otimes \alpha[x_i, y_1, \ldots, y_{n-1}]\otimes \alpha^2(x_{i+1})\otimes \ldots\\
& \hspace{2cm} \alpha^2 x_{j+1} \otimes [\alpha(x_j),\alpha(Z)]\otimes \ldots \otimes \alpha^2(x_{n-1})\\
&+ \sum_{i=1}^{n-1} \alpha^2(x_1) \otimes \ldots \otimes \alpha^2(x_{i-1})\otimes [[x_i, y_1, \ldots, y_{n-1}], \alpha(Z)]\otimes\\
& \hspace{2cm} \alpha^2(x_{i+1})\otimes \ldots \otimes \alpha^2(x_{n-1}).
\end{array}
\end{equation}
And
\begin{equation}\label{sec}
\begin{array}{ll}
&[[X, Z], \alpha(Y)]+ [\alpha(X), [Y, Z]]\\
&=\sum_{i=1}^{n-1}\alpha(x_1)\otimes \ldots \otimes \alpha(x_{i-1}) \otimes [x_i, Z]\otimes \alpha(x_{i+1})\otimes \ldots \otimes \alpha(x_{n-1}), \alpha(Y)]\\
&+[\alpha(X), \sum_{i=1}^{n-1} \alpha(y_1)\otimes \ldots \otimes \alpha(y_{i-1})\otimes [y_i, Z] \otimes \alpha(y_{i+1}) \otimes \ldots \otimes \alpha(y_{n-1})]\\
&= \sum_{j=1}^{i-1}\sum_{i=1}^{n-1} \alpha^2(x_1)\otimes \ldots \otimes \alpha^2(x_{j-1})\otimes [\alpha(x_j), \alpha(Y)]\otimes \ldots \otimes \alpha[x_i, Z]\\
&\hspace{2cm} \otimes \alpha^2(x_{i+1})\otimes \ldots \otimes \alpha^2(x_{n-1})\\
&+ \sum_{j=i+1}^{n-1}\sum_{i=1}^{n-1} \alpha^2(x_1)\otimes \ldots \otimes \alpha^2(x_{i-1})\otimes \alpha[x_i, Z]\otimes \ldots \otimes \alpha[x_j, Y]\\
&\hspace{2cm} \otimes \alpha^2(x_{j+1})\otimes \ldots \otimes \alpha^2(x_{n-1})\\
&+\sum_{i=1}^{n-1} \alpha^2(x_1)\otimes \ldots \otimes \alpha^2(x_{i-1}) \otimes [[x_i, Z], \alpha(Y)] \otimes \ldots \otimes \alpha^2(x_{n-1})\\
&+ \sum_{j=1}^{n-1} \sum_{i=1}^{n-1} \alpha^2(x_1) \otimes \ldots \otimes \alpha^2(x_{j-1}) \otimes [ \alpha(x_j), \alpha(y_1) \otimes \ldots \otimes \alpha(y_{i-1})\otimes\\
&\hspace{1cm}  [y_i, Z] \otimes \alpha(y_{i+1})\otimes \ldots \otimes \alpha(y_{n-1})] \otimes \alpha^2(x_j) \otimes \ldots \otimes \alpha^2(x_{n-1}).
\end{array}
\end{equation}
The first term in (\ref{first}) is same as  the second term in (\ref{sec}), the second term in (\ref{first}) is same as the first term in (\ref{sec}). Using the defining relation of a $n$-Hom-Leibniz algebra, the third term of (\ref{first}) is precisely the sum of last two terms of (\ref{sec}).
This shows that the bracket defined on $D_{n-1}(L)$ endows it with a Hom-Leibniz algebra structure.  It is straight forward to see that $\ba [X, Y]= [\ba (X), \ba (Y)]$. \qed

\section{Cohomology of $n$-Hom-Leibniz algebras}
 In this section, we define a representation of a $n$-Hom-Leibniz algebra and then  a cohomology for $n$-Hom-Leibniz algebras, with coefficients in a representation. We recall first the definition of a representation of a Hom-Leibniz algebra according to \cite{CS}.
\begin{definition}
 A representation of a Hom-Leibniz algebra $(L, [\cdot,\cdot ], \alpha)$ is a pair  $(V, \alpha_V)$ of vector space $V$ and a linear map $\alpha_V:V\rightarrow V$, equipped with $2$ actions 
$$\begin{array}{rcl}
\mu_l:&  L \otimes V &\longrightarrow V\\
\mu_r:&  V\otimes L &\longrightarrow V
\end{array}$$
satisfying for all $x, y \in L$ and $v \in V$
$$\begin{array}{ll}
\mu_l(\alpha(x), \alpha_V(v))&= \alpha_V(\mu_l(x, v))\\
\mu_r(\alpha_V(v), \alpha(x))&= \alpha_V(\mu_r(v, x))
\end{array}$$
and
$$\begin{array}{rcl}
\mu_l([x, y], \alpha_V(v))&=& \mu_r(\mu_l(x, v), \alpha(y) + \mu_r(\alpha(x), \mu_l(y, v))\\
\mu_r(\alpha_V(v), [x, y])&=& \mu_r(\mu_r(v, x), \alpha(y) - \mu_r(\mu_r(v, y), \alpha(x))\\
\mu_l(\alpha(x), \mu_r(v, y))&=& \mu_r(\mu_l(x, v) \alpha(y))- \mu_l([x, y], \alpha_V(v)).
\end{array}
$$
\end{definition}

\begin{definition}\label{repdef}
A representation of a multiplicative $n$-Hom-Leibniz algebra $(L, [\cdot,\cdots, \cdot], \alpha)$ is a pair $(M, \alpha_M)$, where $M$ is a vector space and $\alpha_M:M\rightarrow M$ is a linear map, equipped with $n$ actions \\
$$[\cdots, \cdot, \cdots]_i: L^{\otimes i} \otimes M \otimes L^{\otimes n-1-i} \longrightarrow M,~~ 0\leq i \leq n-1,$$
satisfying $(2n-1)$ equations which are obtained from 
\begin{equation}\label{rep}
[[x_1,\cdots, x_n],\alpha(y_1), \cdots, \alpha(y_{n-1})]\\=
\sum_{i=1}^n [\alpha(x_1), \cdots, \alpha(x_{i-1}), [x_i, y_1, \cdots, y_{n-1}], \alpha(x_{i+1}), \cdots \alpha(x_n)],\end{equation} 
 by letting exactly one of the variables $x_1, x_2, \ldots, x_n, y_1, \cdots, y_{n-1}$ be in $M$ (hence the corresponding $\alpha$ should be replaced by $\alpha_M$) and all others in $L$. 
 
Note that the $n$-Hom-Leibniz algebra $(L, [\cdot,\cdots,\cdot], \alpha)$ is a representation of itself.

For $n=2$ we get back the definition of representation for a Hom-Leibniz algebra.

\end{definition}


\begin{definition}
Let $L= (L, [\cdot, \cdots, \cdot], \alpha)$ be a multiplicative $n$-Hom-Leibniz algebra. We define the $p$-cochains $C^p(L, L)$ of $L$ with coefficients in $L$, $p\geq 1$, as linear maps 
$$f:L \otimes (\mathcal{D}_{n-1}(L))^{\otimes p-1} \longrightarrow L$$ such that 
$\alpha\circ f= f \circ (\alpha \otimes\ba ^{\otimes p-1})$. We define the coboundary map $\delta^p$ from $p$-cochains to $(p+1)$-cochains, for $X_i \in \mathcal{D}_{n-1}(L)$, $1\leq i\leq p$ and $z\in L$, as\\
$$\begin{array}{lll}
&\delta^p(f)(z, X_1, \ldots, X_p)\\
& = \sum_{1\leq i < j}^p (-1)^j f(\alpha(z), \ba (X_1), \ldots, \ldots, \ba (X_{i-1}), [X_i, X_j], \ldots, \hat{X_j}, \ldots, \ba (X_p))\\
&+ \sum_{i=1}^p (-1)^{i} f([z, X_i], \ba (X_1), \ldots, \hat{X_i}, \ldots, \ba (X_p))\\
&+\sum_{i=1}^p (-1)^{i+1} [f(z, X_1, \ldots, \hat{X_i}, \ldots, X_p), \ba ^{p-1}(X_i)]\\
&+ (X_1._\alpha f(~, X_2, \ldots, X_{p} )) .\alpha^{p-1}(z),
\end{array}$$
where $$\begin{array}{rl}& (X_1._\alpha f(~,X_2, \ldots, X_{p} )) .\alpha^{p-1}(z)\\
&= \sum_{i=1}^{n-1}[\alpha^{p-1}(z), \alpha^{p-1}(X_1^1), \ldots, \alpha^{p-1}(X_1^{i-1}), f(X_1^i, X_2, \ldots, X_{p}), \alpha^{p-1}(X_1^{i+1}),\ldots, \alpha^{p-1}(X_1^{n-1})],
\end{array}$$
where $X_i=(X_i^j)_{1\leq j\leq n-1}. $
\end{definition}

\begin{proposition}
The coboundary map $\delta^p$ defined above is a square zero map, i.e., $\delta^{p+1}\circ \delta^p = 0.$
\end{proposition}
{\bf Proof.} The proof will involve five steps. It will be convenient to denote the four terms of $\delta^p$ by 
$\delta^p= \delta^p_1 + \delta^p_2 + \delta^p_3 + \delta^p_4,$ or simply by $ \delta_1 + \delta_2 + \delta_3 + \delta_4.$

{\bf Step I.} Here we show that $\delta^{p+1}_1\delta^p_1=0.$
$$\begin{array}{ll}
&\delta^{p+1}_1\delta^p_1f(z, X_1, \ldots, X_{p+1})\\
&= \sum_{1\leq i<j\leq p+1} (-1)^j \delta_1^p f(\alpha(z), \ba (X_1), \ldots, [X_i, X_j], \ldots, \hat{X_j}, \ldots, \ba (X_{p+1}))\\
&= \sum_{1\leq i<k<j\leq p+1} (-1)^{j+k} f (\alpha^2(z), \ba^2(X_1), \ldots, [[X_i, X_j], \ba (X_k)], \ldots, \hat{\ba (X_k)},\ldots, \hat{X_j},\\
& \ldots, \ba^2(X_{p+1}))\\
&+\sum_{1\leq i<k<j\leq p+1} (-1)^{j+k} f(\alpha^2(z),\ba^2(X_1), \ldots, [\ba (X_i), [X_k, X_j]], \ldots, \hat{[X_k, X_j]},\ldots, \hat{X_j}, \\
&\ldots, \ba^2(X_{p+1}))\\
&+\sum_{1\leq i<k<j\leq p+1} (-1)^{j+k-1} f(\alpha^2(z),\ba^2(X_1), \ldots, [[X_i, X_k], \ba (X_j)], \ldots,\hat{X_j} ,\ldots, \hat{\ba (X_k)}, \\
&\ldots, \ba^2(X_{p+1}))
\end{array}$$
Applying Hom-Leibniz identity $\eqref{homleib}$ to $X_i, X_j, X_k \in \mathcal{D}_{n-1}$, we get $\delta^{p+1}_1\delta^p_1=0.$\\

{\bf Step II.} Here we show that $\delta^{p+1}_1\delta^p_2 +\delta^{p+1}_2\delta^p_1 +\delta^{p+1}_2\delta^p_2 =0$.
$$\begin{array}{ll}
& (\delta^{p+1}_1\delta^p_2 +\delta^{p+1}_2\delta^p_1) f(z, X_1, \ldots, X_{p+1})\\
&= \sum_{1\leq i<j\leq p+1} (-1)^{i+j} f( [\alpha(z), [X_i, X_j]], \ba^2(X_1), \ldots, \hat{[X_i, X_j]}, \ldots, \hat{X_j},\ldots, \ba^2(X_{p+1})).\\
\end{array}
$$
Also, $$\begin{array}{ll}
&\delta^{p+1}_2\delta^p_2 f(z, X_1, \ldots, X_{p+1})\\
&= \sum_{1\leq i<j\leq p+1} (-1)^{i+j}f( [[z, X_j],\ba (X_i)],\ba^2(X_1), \ldots, \hat{\ba (X_j)}, \ldots, \hat{X_j}, \ldots, \ba^2 (X_{p+1}))\\
&+ \sum_{1\leq i<j\leq p+1} (-1)^{i+j-1}f( [[z, X_i],\ba (X_j)], \ba^2(X_1), \ldots, \hat{X_i}, \ldots, \hat{\ba (X_j)}, \ba^2 (X_{p+1}))
\end{array}$$

$\delta^{p+1}_1\delta^p_2 +\delta^{p+1}_2\delta^p_1 +\delta^{p+1}_2\delta^p_2 =0$ by the $n$-Hom-Leibniz algebra identity $\eqref{nhomleib}$.

{\bf Step III.} 
In this step we show that $\delta^{p+1}_1\delta^p_3 +\delta^{p+1}_3\delta^p_1 + \delta^{p+1}_3\delta^p_3=0$.
$$\begin{array}{ll}
&\delta^{p+1}_1\delta^p_3 f(z, X_1, \ldots, X_{p+1})\\
&=\sum_{i<k<j} \{ (-1)^{j+k+1} [f( \alpha(z),\ba(X_1), \ldots, [X_i, X_j], \ldots, \hat{X_k}, \ldots, \hat{X_j}, \ldots, \ba(X_{p+1})), \ba^p(X_k)]\\
&+ (-1)^{j+i+1} [f( \alpha(z), \ba(X_1), \ldots,\hat{\ba(X_i)},\ldots, [X_k, X_j], \ldots, \hat{(X_j)}, \ldots, \ba(X_{p+1})), \ba^p(X_i)]\\
&+(-1)^{j+k-1} [f(\alpha(z), \ba(X_1), \ldots, [X_i, X_k], \ldots, \hat{(X_k)}, \ldots, \hat{\ba(X_j)}, \ldots, \ba(X_{p+1})), \ba^p(X_j)]\}\\
&+ \sum_{i<j} (-1)^{i+j+1} [f(\alpha(z),\ba(X_1), \ldots, \hat{[X_i, X_j]}, \ldots, \hat{X_j}, \ldots, \ba(X_{p+1})), \ba^p[X_i, X_j]]
\end{array}$$

$$\begin{array}{ll}
&\delta^{p+1}_3\delta^p_1 f(z, X_1, \ldots, X_{p+1})\\
&= \sum_{i<k<j} \{(-1)^{j+k}[f(\alpha(z), \ba(X_1), \ldots, [X_i, X_j], \ldots, \hat{X_k},\ldots, \hat{\ba(X_j)},\ldots, \ba(X_{p+1})), \ba^p(X_k)]\\
& (-1)^{j+k} [f(\alpha(z),\ba(X_1), \ldots, [X_i, X_k], \ldots, \hat{X_k}, \ldots, \hat{\ba(X_j)}, \ldots, \ba(X_{p+1})), \ba^p(X_j)]\\
& (-1)^{j+i} [f(\alpha(z),\ba(X_1), \ldots,\hat{X_i},\ldots, [X_k, X_j], \ldots, \hat{\ba(X_j)}, \ldots, \ba(X_{p+1}), ), \ba^p(X_i)]\}.
\end{array}$$

So, $$\begin{array}{ll}
&( \delta^{p+1}_1\delta^p_3 +\delta^{p+1}_3\delta^p_1 ) f(z, X_1, \ldots, X_{p+1})\\
&= \sum_{i<j} (-1)^{i+j+1} [f(\alpha(z), \ba(X_1), \ldots, \hat{[X_i, X_j]}, \ldots, \hat{X_j}, \ldots, \ba(X_{p+1})), \ba^p[X_i, X_j]].
\end{array} $$

Now $$\begin{array}{ll}
& \delta^{p+1}_3\delta^p_3 f(z, X_1, \ldots, X_{p+1})\\
&= \sum_{i<j}\{ (-1)^{i+j}[[f(z, X_1, \ldots, \hat{X_i}, \ldots, \hat{X_j}, \ldots, X_{p+1}), \ba^{p-1}(X_i)], \ba^p(X_j)] +\\
 &(-1)^{i+j+1}[[f(z, X_1, \ldots,  \hat{X_i}, \ldots, \hat{X_j}, \ldots, X_{p+1}), \ba^{p-1}(X_j)], \ba^p(X_i)]\}.
\end{array}
$$
Then
$\delta^{p+1}_1\delta^p_3 +\delta^{p+1}_3\delta^p_1 + \delta^{p+1}_3\delta^p_3=0$ follows from $n$-Hom-Leibniz identity $\eqref{nhomleib}$.\\

{\bf Step-IV.} To show, $\delta^{p+1}_2\delta^p_3 +\delta^{p+1}_3\delta^p_2 =0$.
$$\begin{array}{ll}
&\delta^{p+1}_2\delta^p_3 f(z, X_1, \ldots, X_{p+1})\\
&= \sum_{1\leq i<j\leq p+1}\{ (-1)^{i+j+1} [f([z, X_j],\ba(X_1), \ldots, \hat{\ba(X_i)}, \ldots, \hat{\ba(X_j)}, \ldots, \ba(X_{p+1})), \\
&\ba^p(X_i)]+ (-1)^{i+j} [f([z, X_i],\ba(X_1), \ldots, \hat{\ba(X_i)}, \ldots, \hat{\ba(X_j)}, \ldots, \ba(X_{p+1})), \ba^p(X_j)]\\
&= - \delta^{p+1}_3\delta^p_2 f(X_1, \ldots, X_{p+1},z).
\end{array}$$

{\bf Step V.} Finally, we show that $\delta^{p+1}_2\delta^p_4 + \delta^{p+1}_3\delta^p_4 + \delta^{p+1}_4\delta^p_3 + \delta^{p+1}_4\delta^p_4=0.$
$$\begin{array}{ll}
& \delta^{p+1}_1\delta^p_4 f(z, X_1, \ldots, X_{p+1})\\
&=\sum_{1\leq i<j\leq p} (-1)^j \sum_{1\leq k\leq n-1}  [\alpha^p(z), \ba^p(X_1)^{(1)}, \ldots, f( \ba(X_1)^{(k)},\ba(X_2), \ldots, [X_i, X_j], \\
&\ldots, \hat{\ba(X_j)}, \ldots, \ba(X_{p+1})), \ldots, \ba^p(X_1)^{(n-1)})] \\
&+ \sum_{1\leq j\leq p}(-1)^j \sum_{1\leq k\leq n-1}  [\alpha^p(z), \ba^{p-1}[X_1, X_j]^{(1)}, \ldots, f([X_1, X_j]^{(k)},\ba(X_2), \ldots, \\
&\hat{\ba(X_j)}, \ldots, \ba(X_{p+1}) ), \ldots, \ba^{p-1}[X_1, X_j]^{(n-1)}].
\end{array}$$

$$\begin{array}{ll}
& \delta^{p+1}_4\delta^p_1 f(z, X_1, \ldots, X_{p+1})\\
&= \sum_{1\leq k\leq n-1}  \sum_{2\leq i<j \leq p+1} (-1)^j [\alpha^p(z), \ba^p(X_1)^{(1)}, \ldots, \\
&f( \ba(X_1)^{(k)},\ba(X_2), \ldots, [X_i, X_j], \ldots, \hat{X_j},\ldots, \ba(X_{p+1})), \ldots, \ba^p(X_1)^{(n-1)}].
\end{array}$$

Note that $\delta^{p+1}_4\delta^p_1 f(z, X_1, \ldots, X_{p+1})$ is negative of the first term of $\delta^{p+1}_1\delta^p_4 f(z, X_1, \ldots, X_{p+1})$.
$$\begin{array}{ll}
&\delta^{p+1}_2\delta^p_4 f(z, X_1, \ldots, X_{p+1})\\
&= \sum_{1\leq j\leq n-1} [\ba^{p-1}[z, X_1], \ba^p(X_2)^{(1)}, \ldots, f( \ba(X_2)^{(j)},\ba(X_3), \ldots, \ba(X_{p+1})), \\
&\ldots, \ba^p(X_2)^{(n-1)}]+ \sum_{2\leq i\leq p+1} \sum_{1\leq j \leq n-1}  [\alpha^{p-1}[z, X_i], \ba^p(X_1)^{(1)}, \ldots, f( \ba(X_1)^{(j)},\ba(X_2), \ldots, \\
&\hat{\ba(X_i)}, \ldots, \ba(X_{p+1})), \ldots, \ba^p(X_1)^{(n-1)}].
\end{array}$$

$$\begin{array}{ll}
&\delta^{p+1}_4\delta^p_2 f(z, X_1, \ldots, X_{p+1})\\
&= \sum_{k=1}^{n-1} \sum_{j=2}^{p+1} (-1)^j [\alpha^p(z), \ba^p(X_1)^1, \ldots, f([X_1^k, X_j],\ba(X_2), \ldots, \hat{X_j}, \ldots, \\
&\ba(X_{p+1})), \ldots, \ba^p(X_1)^{n-1}]
\end{array}$$
 Hence, $\delta^{p+1}_4\delta^p_2= -( \delta^{p+1}_1\delta^p_4 + \delta^{p+1}_4\delta^p_1).$

$$\begin{array}{ll}
&\delta^{p+1}_4\delta^p_3 f(z, X_1, \ldots, X_{p+1})\\
&= \sum_{k=1}^{n-1}  \sum_{i=2}^{p+1} (-1)^i [\alpha^p(z), \ba^p(X_1)^{(1)}, \ldots, [f( X_1^{(k)},X_2, \ldots, \hat{X_i}, \ldots, \\
& X_{p+1}), \ba^{p-1}(X_i)], \ldots, \ba(X_1)^{(n-1)}].
\end{array}$$

$$\begin{array}{ll}
&\delta^{p+1}_3\delta^p_4 f(z, X_1, \ldots, X_{p+1})\\
&= \sum_{k=1}^{n-1} [[\alpha^{p-1}(z), \ba^{p-1}(X_2)^{(1)}, \ldots, f( X_2^{(k)}, X_3, \ldots, X_{p+1}), \ldots, \ba^{p-1}(X_2)^{(n-1)}], \ba^p(X_1)]\\
&+ \sum_{i=2}^{p+1} (-1)^{i+1} \sum_{k=1}^{n-1}(-1)^{p}[[\alpha^{p-1}(z), \ba^{p-1}(X_1)^{(1)}, \ldots, f(X_1^{(k)}, X_2, \ldots, \hat{X_i}, \ldots, X_{p+1}), \ldots,\\
& \ba^{p-1}(X_1)^{(n-1)}], \ba^p(X_i)].
\end{array}$$

$$\begin{array}{ll}
&\delta^{p+1}_4\delta^p_4 f(z, X_1, \ldots, X_{p+1})\\
&= \sum_{k=1}^{n-1} [\alpha^p(z), \ba^p(X_1)^{(1)}, \ldots, \sum_{j=1}^{n-1} (-1)^{p}[ \alpha^{p-1}(X_1)^{(k)}, \ba^{p-1}(X_2)^{(1)}, \\
&\ldots, f( X_2^{(j)},X_3, \ldots, X_{p+1}), \ldots, \ba^{p-1}(X_2)^{(n-1)}], \ldots, \ba^p(X_1)^{(n-1)}].
\end{array}$$

By the $n$-Hom-Leibniz identity, it follows that  
(First part of $\delta^{p+1}_2\delta^p_4$) + (first part of $\delta^{p+1}_3\delta^p_4 ) = -\delta^{p+1}_4\delta^p_4,$ and (second part of $\delta^{p+1}_2\delta^p_4)+ \delta^{p+1}_4\delta^p_3=$--(second part of $\delta^{p+1}_3\delta^p_4$).   
Hence, $$\delta^{p+1}_2\delta^p_4 + \delta^{p+1}_3\delta^p_4 + \delta^{p+1}_4\delta^p_3 + \delta^{p+1}_4\delta^p_4=0.$$ This completes the proof of the theorem.
\qed

Therefore, we get a cohomology complex with values in the algebra.
\begin{definition}
Let the space of $p$-cocycles of the cohomology complex $C^*(L, L)$ be denoted by
$\mathcal{Z}^p(L, L)$ and the space of $p$-coboundaries of $C^*(L, L)$ be denoted by
$\mathcal{B}^p(L, L)$. We define the $p$th cohomology group of the $n$-ary multiplicative Hom-Leibniz algebra $L$, with coeffiecients in itself as the quotient group
$$H^p(L, L)= \mathcal{Z}^p(L, L)/ \mathcal{B}^p(L, L).$$
\end{definition}

\subsection{Cohomology with coefficients in a representation}
We analogously define the cohomology of a multiplicative $n$-Hom-Leibniz algebra $(L, [., \ldots, .], \alpha)$ with coefficients in a representation $(M, \alpha_M)$ as follows.

\begin{definition}\label{rep}
Let $(M, \alpha_M)$ be a representation of the multiplicative $n$-Hom-Leibniz algebra $(L, [\cdot, \cdots, \cdot], \alpha)$. We define the $p$-cochains $C^p(L, M)$ of $L$ with coefficients in $M$ as linear maps 
$$f: L\otimes \mathcal{D}_{n-1}(L)^{\otimes p-1} \longrightarrow M$$ such that 
$\alpha_M\circ f= f \circ (\alpha \otimes \ba ^{\otimes p-1})$. We define the coboundary map $\delta^p: C^p(L, M)\longrightarrow C^{p+1}(L, M)$ by
$$\begin{array}{lll}
&\delta^p(f)(z, X_1, \ldots, X_p)\\
&= \sum_{1\leq i < j}^p (-1)^j f(\alpha(z), \ba (X_1), \ldots, \ba (X_{i-1}), [X_i, X_j],  \ba (X_{i+1}), \ldots, \hat{X_j}, \ldots, \ba (X_p))\\
&+ \sum_{i=1}^p (-1)^i f([z, X_i], \ba (X_1), \ldots, \hat{X_i}, \ldots, \ba (X_p))\\
&+\sum_{i=1}^p (-1)^{i+1} [f(z, X_1, \ldots, \hat{X_i}, \ldots, X_p), \ba ^{p-1}(X_i)]_0\\
 &+ (X_1._\alpha f(~, X_2, \ldots, X_{p} )) .\alpha^{p-1}(z),
\end{array}$$
where $$\begin{array}{rl}& (X_1._\alpha f(~,X_2, \ldots, X_{p} )) .\alpha^{p-1}(z)\\
&= \sum_{i=1}^{n-1}[\alpha^{p-1}(z), \alpha^{p-1}(X_1^{1}), \ldots, f(X_1^{i}, X_2, \ldots, X_{p}), \ldots, \alpha^{p-1}(X_1^{n-1})]_i,
\end{array}$$
and $X_i=(X_i^{j})_{1\leq j\leq n-1}. $

 The last two terms in the definition of the coboundary make use of the $n$ actions $$[\cdots, \cdot, \cdots]_i: L^{\otimes i} \otimes M \otimes L^{\otimes n-1-i} \longrightarrow M,~~ 0\leq i \leq n-1.$$

\end{definition}

\begin{proposition}
The map $\delta^p$ defined in Definition  $\ref{rep}$ is a square zero map. 
\end{proposition}
{\bf Proof.} The proof is similar to the proof of the previous proposition, and uses the defining relations of a $n$-Hom-Leibniz algebra representation  $\eqref{repdef}$. \qed

\section{Interpretation of first and second cohomology groups}

\subsection{\bf First Cohomology group $H^1(L; M)$}

Let $\mathrm{Der}(L, M)$ denote the vector space of linear maps $f: L\longrightarrow M$ such that
$$f([x_1, \ldots, x_n])= \sum_{i=0}^{n-1}[ x_1, \ldots, x_{i-1}, f(x_i), x_{i+1},\ldots, x_{n}]_i.$$ Elements of $\mathrm{Der}(L, M)$ are called derivations of $L$ with values in the representation $M$. 
Note that $H^1(L, M)\cong \mathcal{Z}^1(L, M)$. Let $f \in \mathcal{Z}^1(L, M)$. For $z \in L$ and $X=(x_1, \ldots, x_{n-1})\in \mathcal{D}_{n-1}$, we have
$$\begin{array}{ll}
0&=\delta^1f(z, X)\\
&= - f([z, X]) + [f(z), X]_0 + \sum_{i=1}^{n-1}[z, x_1, \ldots, x_{i-1}, f(x_i), \ldots, x_{n-1}]_i.
\end{array}$$
This simply means that $f$ is a derivation of $L$ with values in $M$ if and only if $f \in \mathcal{Z}^1(L, M)$. This proves that $H^1(L, M)\cong \mathcal{D}er(L, M)$.

\subsection{\bf Second cohomology group $H^2(L, M)$.}
 An abelian extension of a $n$-Hom-Leibniz algebra is a ($\mathbb{K}$-split) exact sequence of $n$-Hom-Leibniz algebras 
\begin{equation}\label{ext}
0\longrightarrow (M, \alpha_M) \stackrel{\iota}{\longrightarrow} (K, \alpha_K) \stackrel{\pi}{\longrightarrow}(L, \alpha)\longrightarrow 0,
\end{equation} such that for $a_1, \cdots, a_n \in K$, $[a_1, \cdots, a_n]=0,$ whenever $a_i= \iota(m_i)$ and $a_j=\iota(m_j)$ for some $1\leq i, j \leq n$, $i\neq j$. Hence, the bracket on $M$ is zero. 

Given an abelian extension \eqref{ext} of $n$-Hom-Leibniz algebras, $M$ is equipped with $n$ actions 
$$[\cdot, \cdots, \cdot]_i: L^{\otimes i} \otimes M \otimes L^{\otimes n-1-i}\longrightarrow M,$$ for as follows. For $0 \leq i \leq n-1$, define 
$$[x_1, \ldots, x_i, m, x_{i+1}, \ldots, x_{n-1}]_i= [sx_1, \ldots, sx_i, \iota m, sx_{i+1}, \ldots, sx_{n-1}],$$ where the linear map $s$ is the splitting of the exact sequence 
$$0\longrightarrow (M, \alpha_M) \stackrel{\iota}{\longrightarrow} (K, \alpha_K) \stackrel{\pi}{\longrightarrow}(L, \alpha)\longrightarrow 0.$$
These $n$ actions so defined satisfies the $(2n-1)$ relations as in equation \eqref{repdef}. In other words, $(M, \alpha_M)$ inherits the structure of a representation of $(L, \alpha)$. 

Now, let us start with a $n$-Hom-Leibniz algebra $(L, \alpha)$ and let $(M, \alpha_M)$ be a representation of $(L, \alpha)$. Let 
\begin{equation}\label{K} 0\longrightarrow (M, \alpha_M) \stackrel{\iota}{\longrightarrow} (K, \alpha_K) \stackrel{\pi}{\longrightarrow}(L, \alpha)\longrightarrow 0
\end{equation} 
be an abelian extension, such that the induced structure of representation of $L$ on $M$ induced by the extension coincides with the original one. If this condition holds, we call \eqref{K} an abelian extension of $L$ by $M$. Two such extensions 
$$\begin{array}{llll}
0\longrightarrow &(M, \alpha_M) \stackrel{\iota_1}{\longrightarrow}& (K, \alpha_K) \stackrel{\pi_1}{\longrightarrow}&(L, \alpha)\longrightarrow 0\\
0\longrightarrow &(M, \alpha_M) \stackrel{\iota_2}{\longrightarrow}& (K', \alpha_{K'}) \stackrel{\pi_2}{\longrightarrow}&(L, \alpha)\longrightarrow 0
\end{array}$$
are said to be isomorphic if there exists a $n$-Hom-Leibniz algebra homomorphism $\phi: K\longrightarrow K'$ such that the following diagram commutes.
\newcommand{\Lrightarrow}{\hbox to1cm{\rightarrowfill}}
\newcommand{\Ldownarrow}{\bigg\downarrow}

\[
  \setlength{\arraycolsep}{1pt}
  \begin{array}{*{9}c}
    0 &\Lrightarrow & (M, \alpha_M) & \stackrel{\iota_1}{\Lrightarrow} & (K, \alpha_K) & \stackrel{\pi_1}{\Lrightarrow} & (L, \alpha) & \Lrightarrow & 0\\
    & & \Ldownarrow\mbox{Id}_M & & \Ldownarrow \phi & & \Ldownarrow \mbox{Id}_L& & \\
 0 &\Lrightarrow & (M, \alpha_M) & \stackrel{\iota_2}{\Lrightarrow} & (K', \alpha_{K'}) & \stackrel{\pi_2}{\Lrightarrow} & (L, \alpha) & \Lrightarrow & 0    
  \end{array}
\]

 ie. $\phi \iota_1= \iota_2$ and $\pi_2 \phi=\pi_1$.  Let $\mathrm{Ext}(L, M)$ be the set of isomorphism classes of extensions of $L$ by $M$. 

Let $f: L \otimes L^{\otimes n-1} \longrightarrow M$ be a linear map. We define a $n$-bracket on $K= M\oplus L$ by 
$$\begin{array}{ll}
&[(m_1, x_1), (m_2, x_2), \ldots, (m_n, x_n)]\\
&= (\sum_{i=1}^n[x_1, \ldots, m_i, \ldots, x_n]_i + f(x_1, \ldots, x_n), [x_1, \ldots, x_n]),
\end{array}$$ for $m_i \in M$ and $x_i \in L$ for $1\leq i \leq n$. Also, let $\alpha_K(m, x)=(\alpha_M(m), \alpha (x))$.
Note that, $K$ with the bracket induced by $f$ as above, and with $\alpha_K$ as defined, is a $n$-Hom-Leibniz algebra if and only if 
$$\begin{array}{ll}
&f([z, X], \ba(Y)) + [f(z, X), \ba(Y)]\\
&= f([z, Y], \ba(X)) + \sum_{i=1}^{n-1} f(\alpha(z), x_1, \ldots, x_{i-1},[x_i, Y], x_{i+1} \ldots, x_{n-1}) \\
&+[f(z, Y), \ba(X)] + \sum_{i=1}^{n-1} [\alpha(z), \alpha(x_1), \ldots, \alpha(x_{i-1}), f(x_i, \ba(Y)), \ldots, \alpha(x_n)], 
\end{array}$$ that is, $f$ is a $2$-cocycle. This abelian extension 
$$0 \Lrightarrow  (M, \alpha_M)  \stackrel{\iota}{\Lrightarrow}  (K, \alpha_K) \stackrel{\pi}{\Lrightarrow}  (L, \alpha) \Lrightarrow  0$$ where the $n$-Hom-Leibniz algebra bracket on $K=M\oplus L$ is induced by the $2$-cocycle $f$, is denoted by $K_f$.

\begin{lemma} With the above notation, $f$ and $g$ determine the same cohomology class in $H^2(L,M)$ if and only if $K_f$ and $K_g$ are isomorphic. 
\end{lemma}
{\bf Proof.} Let $f$ and $g$ determine the same cohomology class in $H^2(L,M)$, ie. there exists a $1$-cochain, say $h$, such that $f-g= \delta^1h$. We define a bijective map $\phi: K_f \longrightarrow K_g$ by\\
$\phi(m, x)= (m+h(x), x)$, for $m\in M$ and $x\in L$. It is straight forward to verify that $\phi$ so defined is a $n$-Hom-Leibniz algebra map from $K_f$ to $K_g$, which commutes with identity on $M$ and $L$. Conversely, let $\eta: K_f\longrightarrow K_g$ be a $n$-Hom-Leibniz algebra equivalence. Since $\eta$ commutes with identities on $M$ and $L$, $\eta(m, x)= (m+ r(x), x)$, where $r$ is a map from $L$ to $M$. Using the $n$-Hom-Leibniz algebra structure on $K_f$ and $K_g$, and the fact that $\eta$ is a $n$-Hom-Leibniz algebra map, it is routine to check that $r$ is a $1$-cochain such that $f-g= \delta^1(r)$.  \qed

This defines a map 
$$\begin{array}{ll}
\Psi: & H^2(L, M)\longrightarrow Ext(L, M)~~ \mbox{by}\\
& <f> \mapsto K_f
\end{array}.$$ 
It is also straightforward to check that $\Psi$ so defined is a bijection.
Hence, $H^2(L,M) \cong Ext(L,M)$.

\section{Embedding the cochain complex $C^*(L,L)$ into $C^*(\mathcal{D}_{n-1}(L), \mathcal{D}_{n-1}(L))$}

It turns out that the space of cochains for a $n$-Hom-Leibniz algebra $(L, [\cdot, \cdots, \cdot], \alpha)$ can be embedded into the space of cochains for its associated Hom-Leibniz algebra $\mathcal{D}_{n-1}(L)$, if $\alpha$ is an injective map. At the same time, the coboundary map is also preserved by this embedding. For the sake of completeness, let us recall here the cochain complex of a Hom-Leibniz algebra $\mathfrak{g}$ with coefficients in a representation $V$.

Let $(\mathfrak{g}, [\cdot, \cdot], \alpha)$ be a Hom-Leibniz algebra and $(V,\alpha_V)$ be a representation of $(\mathfrak{g}, [\cdot, \cdot], \alpha)$. For $n \geq 1$, let  
$C^n(\mathfrak{g}, V)$ denote the space of all $f \in Hom(\mathfrak{g}^{\otimes n}, V )$ such that $f(\alpha^{\otimes n})= \alpha_V f$.
Define the coboundary operator as
$d^n: C^n(\mathfrak{g}, V) \longrightarrow C^{n+1}(\mathfrak{g}, V)$ 
by
$$\begin{array}{lll}
(d_nf)(x_1\otimes \cdots \otimes x_{n+1})&=&[\alpha^{n-1}(x_1), f(x_2 \otimes \cdots \otimes x_{n+1})]\\
&+& \sum_{i=2}^{n+1}(-1)^i [f(x_1\otimes \cdots \otimes \check{x_i}\otimes \cdots \otimes x_{n+1}), \alpha^{n-1}(x_i)]\\
&+&\sum_{1\leq i<j\leq n+1}(−1)^{j+1}f(\alpha (x_1) \otimes \cdots \otimes \alpha(x_{i-1}) \otimes [x_i, x_j ] \otimes\\ &&\alpha(x_{i+1})\otimes \cdots \otimes \check{x_j} \otimes \cdots \otimes \alpha(x_{n+1})).\end{array}$$
We refer to \ref{CS} for a proof of the fact that the above defined operator is a square zero map. The resulting cochain complex $(C^*(\mathfrak{g}, V), d)$ is called the Chevalley-Eilenberg complex for the Hom-Leibniz algebra $(\mathfrak{g}, [\cdot, \cdot], \alpha)$ with coefficients in a representation $(V, \alpha_V)$.

Let $(L, [\cdot, \cdots, \cdot], \alpha)$ be a $n$-Hom-Leibniz algebra. 
Let $\Delta: Hom(L, L)\longrightarrow Hom(\mathcal{D}_{n-1}(L), \mathcal{D}_{n-1}(L))$ be given by:
$$\Delta (f)(x_1\otimes \ldots \otimes x_{n-1})= \sum_{i=1}^{n-1}\alpha(x_1) \otimes \ldots \otimes \alpha(x_{i-1}) \otimes f(x_i)\otimes \alpha(x_{i+1})\otimes \ldots \otimes \alpha(x_{n-1}).$$
This map $\Delta$ induces a map, again denoted by $\Delta: C^p(L, L) \longrightarrow C^p(\mathcal{D}_{n-1}(L), \mathcal{D}_{n-1}(L))$ given by
$$\begin{array}{ll}
&\Delta(f)(X_1 \otimes \ldots \otimes X_p)\\
&= \sum_{i=1}^{n-1}\alpha^{p-1}(X_1^1)\otimes\ldots \otimes \alpha^{p-1}(X_1^{i-1}) \otimes f(X_1^i, X_2\otimes \ldots \otimes X_{p})\otimes \alpha^{p-1}(X_1^{i+1})\otimes
\ldots \otimes \alpha^{p-1}(X_1^{n-1}).
\end{array}$$ 

The following theorem shows that the map  $\Delta: C^p(L, L) \longrightarrow C^p(\mathcal{D}_{n-1}(L), \mathcal{D}_{n-1}(L))$ is a map of cochain complexes. 

\begin{theorem}
The map $\Delta: C^p(L, L) \longrightarrow C^p(\mathcal{D}_{n-1}(L), \mathcal{D}_{n-1}(L))$ as defined above commutes with the coboundaries of the cochain complexes. In other words, the following diagram commutes.

\[
\xymatrix{
C^p(\mathcal{D}_{n-1}(L), \mathcal{D}_{n-1}(L)) \ar[r]^{d^p}  & C^{p+1}(\mathcal{D}_{n-1}(L), \mathcal{D}_{n-1}(L)) \\
  C^p(L, L) \ar[r]^{\delta^p} \ar[u]^\Delta  & C^{p+1}(L, L) \ar[u]^\Delta
}
\] where $d^p$ is the $p$th coboundary map of the cochain complex $C^*(\mathcal{D}_{n-1}(L), \mathcal{D}_{n-1}(L))$, \cite{CS}.
\end{theorem}

{\bf Proof.} 
\begin{equation}
\begin{array}{ll}\label{delta}
&(d^p \Delta f)(X_1, \ldots, X_{p+1})\\
&= [\ba^{p-1}(X_1), \Delta f(X_2, \ldots, X_{p+1})] + \sum_{i=2} ^{p+1} (-1)^{i} [ \Delta f (X_1, \ldots, \hat{X_i}, \ldots, X_{p+1}), \ba^{p-1}(X_i)]\\
& + \sum_{1\leq i < j\leq p+1} (-1)^{j+1} \Delta f(\ba(X_1), \ldots, \ba(X_{i-1}), [X_i, X_j], \ldots, \hat{X_j}, \ldots , \ba (X_{p+1}))\\

\end{array}
\end{equation}

The second term of equation\, \eqref{delta} is
$$\begin{array}{ll}
=& \sum_{i=2}^{p+1} (-1)^i [ \sum_{j=1}^{n-1} \alpha^{p-1}(X_1^1) \otimes \ldots \otimes \alpha^{p-1}(X_1^{j-1}) \otimes f(X_1^j, X_2, \ldots \hat{X_i}, \ldots, X_{p+1})\\
& \otimes \alpha^{p-1}(X_1^{j+1})\otimes \ldots \otimes \alpha^{p-1}(X_1^{n-1}), \ba^{p-1}(X_i)]
\end{array}
$$
The above can be further simplified as
\begin{equation}
\begin{array}{ll}\label{delta1}
=&\sum_{i=2}^{p+1} (-1)^{i} \sum_{j=1}^{n-1} \sum_{k<j} \alpha^p(X_1^1)\otimes \ldots \otimes [\alpha^{p-1} (X_1^k), \ba^{p-1}(X_i)] \otimes \ldots \otimes \\
& \alpha f(X_1^j, X_2, \ldots, \hat{X_i}, \ldots, X_{p+1})\otimes \alpha^p(X_1^{j+1})\otimes \ldots \otimes \alpha^p (X_1^{n-1})\\
&+ \sum_{i=2}^{p+1} (-1)^{i} \sum_{j=1}^{n-1} \sum_{k>j} \alpha^p(X_1^1)\otimes \ldots \otimes \alpha^p (X_1^{j-1})\otimes \alpha f( X_1^j, X_2, \ldots, \hat{X_i}, \ldots, X_{p+1})\\
&\otimes \ldots \otimes [\alpha^{p-1}(X_1^k), \ba^{p-1}(X_i)] \otimes \ldots \otimes \alpha^p(X_1^{n-1})\\
&+ \sum_{i=2}^{p+1} (-1)^{i} \sum_{j=1}^{n-1} \alpha^p(X_1^1)\otimes \ldots \otimes \alpha^p(X_1^{j-1}) \otimes [f(X_1^j, X_2, \ldots, \hat{X_i}, \ldots, X_{p+1}), \ba^{p-1}(X_i)]\\
&\otimes \alpha^p(X_1^{j+1}) \otimes \ldots \otimes \alpha^p (X_1^{n-1}).
\end{array}
\end{equation}

As $[X_1, X_j]= \sum_{k=1}^{n-1} \alpha X_1^1 \otimes \otimes \alpha X_1^{n-1} \otimes [X_1^k, X_j] \otimes \ldots \otimes \alpha X_1^{n-1}$, the third term of \eqref{delta} for $i=1$ can be simplified as
\begin{equation}
\begin{array}{ll}\label{delta2}
&\sum_{2 \leq j \leq p+1} (-1)^{j+1} \sum_{k=1}^{n-1} \sum_{m<k} \alpha^p(X_1^1) \otimes \ldots \otimes f(\alpha X_1^m, \ba X_2, \ldots \hat{X_j}, \ldots , \ba(X_{p+1}))\\
&\otimes \ldots \otimes \alpha^{p-1}[X_1^k, X_j] \otimes \ldots \otimes \alpha^p(X_1^{n-1})\\
& +\sum_{2 \leq j \leq p+1} (-1)^{j+1} \sum_{k=1}^{n-1} \sum_{m>k} \alpha^p(X_1^1) \otimes \ldots \otimes \alpha^{p-1}[X_1^k, X_j] \otimes \ldots \otimes\\
& f(\alpha(X_1^m), \ba(X_2), \ldots \hat{X_j}, \ldots , \ba(X_{p+1}))\otimes \ldots \otimes \alpha^p(X_1^{n-1})\\
& + \sum_{k=1}^{n-1} \alpha^p(X_1^1) \otimes \ldots \otimes f([X_1^k, X_j], \ba(X_2), \ldots, \hat{X_j}, \ldots, \ba(X_{p+1}))\otimes \ldots \otimes \alpha^p(X_1^{n-1}).
\end{array}
\end{equation}

Adding \eqref{delta1} and \eqref{delta2}, we get 
$$\begin{array}{ll}
& \sum_{i=2}^{p+1} (-1)^{i} \sum_{j=1}^{n-1} \alpha^p(X_1^1)\otimes \ldots \otimes \alpha^p(X_1^{j-1}) \otimes [f(X_1^j, X_2, \ldots, \hat{X_i}, \ldots, X_{p+1}), \ba^{p-1}(X_i)]\\
&\otimes \alpha^p(X_1^{j+1}) \otimes \ldots \otimes \alpha^p(X_1^{n-1})\\
& + \sum_{j=2}^{p+1} (-1)^{j+1} \sum_{k=1}^{n-1} \alpha^p(X_1^1) \otimes \ldots \otimes f([X_1^k, X_j], \ba(X_2), \ldots, \hat{X_j}, \ldots, \ba(X_{p+1}))\\
&\otimes \alpha^p(X_1^{k+1}) \otimes \ldots \otimes \alpha^p(X_1^{n-1}).
\end{array}
$$

Hence, 
\begin{equation}
\begin{array}{ll}\label{delta3}
&(d^p \Delta f)(X_1, \ldots, X_{p+1})\\
&= [\alpha^{p-1}(X_1), \sum _{i=1}^{n-1} \alpha^{p-1}(X_2^i) \otimes \ldots \otimes \alpha^{p-1}(X_2^{i-1})\otimes f(X_2^i, X_3, \ldots, X_{p+1}) \otimes \alpha^{p-1}(X_2^{i+1})\\
&  \otimes \ldots \otimes \alpha^{p-1}(X_2^{n-1})]+ \sum_{i=2}^{p+1} (-1)^{i} \sum_{j=1}^{n-1} \alpha^p(X_1^1)\otimes \ldots \otimes \alpha^p(X_1^{j-1})\otimes \\
&[f(X_1^j, X_2, \ldots, \hat{X_i}, \ldots, X_{p+1}), \ba^{p-1}(X_i)] \otimes \alpha^p(X_1^{j+1}) \otimes \ldots \otimes \alpha^p(X_1^{n-1})\\
&+\sum_{j=2}^{p+1} (-1)^{j+1} \sum_{k=1}^{n-1} \alpha^p(X_1^1) \otimes \ldots \otimes f([X_1^k, X_j], \ba(X_2), \ldots, \hat{X_j}, \ldots, \ba(X_{p+1}))\\
&\otimes \alpha^p(X_1^{k+1}) \otimes \ldots \otimes \alpha^p(X_1^{n-1})\\
&+\sum_{2 \leq i < j \leq p+1} (-1)^{j+1} \sum_{k=1}^{n-1} \alpha^p(X_1^1) \otimes \ldots \otimes \alpha^{p} (X_1^{k-1})\otimes\\
& f(\alpha(X_1^k), \ba(X_2), \ldots, [X_i, X_j], \ldots, \hat{X_j}, \ldots , \ba(X_{p+1}))\otimes \alpha^p(X_1^{k+1}) \otimes \ldots \otimes \alpha^p(X_1^{n-1}).
\end{array}
\end{equation}

On the other hand, 
\begin{equation}
\begin{array}{ll}\label{delta4}
&\Delta \delta^pf(X_1, \ldots, X_{p+1})\\
= & \sum_{i=1}^{n-1} \alpha^p(X_1^1) \otimes \ldots \otimes \alpha^p(X_1^{i-1})\otimes \delta^p f(X_1^i, X_2, \ldots, X_{p+1})\\
& \otimes \alpha^p(X_1^{i+1})\otimes \ldots \otimes \alpha^p(X_1^{n-1})\\
=& \sum_{i=1}^{n-1} \alpha^p(X_1^1)\otimes \ldots \otimes \alpha^p(X_1^{i-1})\otimes \{\sum_{2\leq j < k}^{p+1}(-1)^k f (\alpha(X_1^i), \ba(X_2), \ldots, \ba(X_{j-1}), [X_j, X_k], \ldots, \hat{X_k}, \ldots, \ba(X_{p+1}))\\
& +\sum_{j=2}^p (-1)^j f([X_1^i, X_j], \ba(X_2), \ldots, \hat{X_j},\ldots ,\ba(X_{p+1}))\\
& + \sum_{j=2}^{p+1} (-1)^{j+1} [f(X_1^i, X_2, \ldots, \hat{X_j}, \ldots, X_{p+1}), \ba^{p-1}(X_j)]\\
&+  \sum_{j=1}^{n-1} [\alpha^p(X_1^i), \alpha^{p-1}(X_2^1)\otimes \ldots \otimes f(X_2^j, X_3, \ldots, X_{p+1})\otimes \ldots \otimes \alpha^{p-1}(X_2^{n-1})]\}\\
&\otimes \alpha^p(X_1^{i+1})\otimes \ldots \otimes \alpha^p(X_1^{n-1}).
\end{array}
\end{equation}

It is easy to see that the right hand side of equation \eqref{delta3} and equation \eqref{delta4} are the same. Hence the theorem follows. \qed

The following lemma is easy to see.
\begin{lemma}
Let $(L, [\cdot, \cdots, \cdot), \alpha)$ be a $n$-Hom-Leibniz algebra. Assume that the map $\alpha:L\longrightarrow L$ is an injective map. Then the map of cochain complexes $\Delta: C^*(L, L) \longrightarrow C^*(\mathcal{D}_{n-1}(L), \mathcal{D}_{n-1}(L))$ is an embedding.
\end{lemma}

Henceforth we shall assume that the map $\alpha: L\longrightarrow L$ is an injective map.

\section{Higher operations on the cochain complex $C^*(L, L)$}

In \cite{B}, D. Balavoine has defined a Lie algebra structure on the cochain complex of a (right) Leibniz algebra. On the same lines we define a bracket operation on the cochain complex of a Hom-Leibniz algebra, which makes this cochain complex into a Lie algebra. 

Let $(\mathfrak{g}, [\cdot, \cdot], \alpha)$ be a Hom-Leibniz algebra. We define a bracket operation on the cochain complex $(C^*(\mathfrak{g}, \mathfrak{g}), d)$ as follows.

Let $M^p(\g)= \mbox{Hom}( \g^{\otimes p+1}, \g)= C^{p+1}(\g, \g)$. Let $f\in M^p(\g)$ and $g \in M^q(\g)$. For $0\leq k \leq p$, and $\sigma$ a $(q, p-k)$ shuffle, define 
$$\circ_k^\sigma: M^p(\mathfrak{g})\otimes M^q(\mathfrak{g})\longrightarrow M^{p+q}(\mathfrak{g})$$ by
$$\begin{array}{ll}
f \circ_k ^\sigma g( x_0\otimes  x_1 \otimes \cdots \otimes x_{p+q})\\
 = f(\alpha^q x_0 \otimes \alpha^q x_1 \otimes \cdots \otimes \alpha^qx_{k-1} \otimes g(x_k \otimes x_{\sigma(k+1)}\otimes \cdots \otimes x_{\sigma(k+q)})\otimes \\
\hspace{1cm} \alpha^q x_{\sigma(k+q+1)}\otimes \cdots \otimes \alpha^q x_{\sigma(p+q)}).
\end{array}
$$ 

We define 
$$i_f(g):= \sum_{k=0}^p (-1)^{kq} \sum_{\sigma \in Sh(q, p-k)} \epsilon(\sigma) f \circ_k ^\sigma g.$$
Next define the bracket as
$$[f, g]_L= i_f(g) + (-1)^{pq+1} i_g(f).$$

Proof of the following lemma is similar to the proof in the Leibniz algebra case, which appears as Lemma\,3.2.6 in \cite{B}.

\begin{lemma}
Let $\g$ be a Hom-Leibniz algebra. The bracket $[.,.]_L$ as defined above makes the graded vector space $M^*(\g)$ into a graded Lie algebra.
\end{lemma}

Next we define a Lie bracket on the cochain complex of a $n$-Hom-Leibniz algebra $L$ with coefficients in itself. We proceed as follows. Let $(L,[\cdot,\cdots,\cdot ], \alpha)$ be a $n$-Hom-Leibniz algebra. Let $\Gamma^p(L)= C^{p+1}(L, L)= \mbox{Hom}(L\otimes \mathcal{D}_{n-1}^{\otimes p}, L).$ Let $f \in \Gamma^p(L)$ and $g \in \Gamma^q(L)$. 
Let $\sigma$ be a $(q, p-k)$ shuffle. For $k=0$ define 
$$
f \circ_0^\sigma g(z, X_1, \cdots , X_{p+q})
 = f(g(z, X_{\sigma(1)}, \cdots, X_{\sigma(q)}), \ba^qX_{\sigma(q+1)}, \cdots, \ba^q X_{\sigma(p+q)}).
$$

For $1\leq k \leq p$ define
$$\begin{array}{ll}
&f \circ_k^\sigma g (z, X_1, \cdots, X_{p+q})
= \sum_{s=1}^{n-1} f(\alpha^qz, \ba^q X_1, \cdots, \ba^q X_{k-1}, \alpha^qX_k^1, \cdots, \alpha^q X_k^{s-1}, \\
& \hspace{1cm} g(X_k^s, X_{\sigma(k+1)}, \cdots, X_{\sigma(k+q)}), \alpha^q X_k^{s+1}, \cdots, \alpha^q X_k^{n-1}, \ba^q X_{\sigma(k+q+1)}, \cdots, \ba^q X_{\sigma(p+q)}).
\end{array}$$

Define
$$i_f(g)= \sum_{k=0}^p (-1)^{kq} \sum_{\sigma \in Sh(q, p-k)} \epsilon(\sigma) f \circ_k^\sigma g.$$

Finally, for $f \in \Gamma^p(L)$ and $g \in \Gamma^q(L)$. we define the bracket $[\cdot, \cdot]_\mathcal{N}$ as 
$$[f, g]_\mathcal{N}= i_f(g) + (-1)^{pq+1} i_g(f).$$

Now, we introduce a distinguised $2$-cochain $\pi$ of $L$, which satisfies the property that $[\pi, \pi]_\mathcal{N}=0$ and the coboundary operator is just bracketting with $\pi$ (with a sign), ie. $\delta(f)= -[f, \pi]_\mathcal{N}.$
\begin{proposition}\label{pipi}
Define a $2$-cochain $\pi$ as 
$\pi(z, X)= [z, X^1, \ldots, X^{n-1}]$. Then $\pi$ is a $2$-cocycle and the equation $[\pi, \pi]_\mathcal{N}=0$ reads as the fundamental identity of the $n$-Hom-Leibniz algebra \eqref{nhomleib}.
\end{proposition}
{\bf Proof.} The fact that $\pi$ so defined is a cocycle follows from the $n$-Hom-Leibniz identity. 
Now, $[\pi, \pi]_\mathcal{N}=2 i_\pi(\pi).$ We show that $i_\pi (\pi)=0$.
$$\begin{array}{ll}
& i_\pi (\pi)(z, X_1, X_2)\\
&= \{\sum_{\sigma \in Sh(1, 1)} \epsilon(\sigma) \pi \circ_0^\sigma \pi -\sum_{\sigma \in Sh(1, 0)}\epsilon(\sigma) \pi\circ_1 ^\sigma \pi\}(z, X_1, X_2)\\
&=\pi(\pi(z, X_1), \alpha X_2)-\pi(\pi(z, X_2),\alpha X_1) - \sum_{s=1}^{n-1} \pi(\alpha z, \alpha X_1^1, \cdots, \alpha X_1^{s-1}, \pi(X_1^s, X_2), \\
&\alpha X_1^{s+1}, \cdots, \alpha  X_1^{n-1})\\
&=0
\end{array}$$ by the $n$-Hom-Leibniz algebra identity \eqref{nhomleib}.

\begin{proposition}\label{pi}
Let $\pi$ be the $2$-cochain of the $n$-Hom-Leibniz algebra $L$ defined as in Proposition \ref{pipi}.  Let $f$ be a $p$-cochain of $L$. Then 
$$-\delta^p (f)= [f, \pi]_\mathcal{N}.$$
\end{proposition}
{\bf Proof.} $$[f, \pi]_\mathcal{N}= i_f(\pi) + (-1)^p i_\pi (f).$$
Now, $$\begin{array}{ll} 
&i_f(\pi)(z, X_1, \cdots, X_p)\\
&= \sum_{k=0}^{p-1} (-1)^k \sum_{\sigma \in Sh(1, p-k-1)} \epsilon(\sigma) f \circ_k ^\sigma \pi (z, X_1, \cdots, X_p)\\
&= \sum_{\sigma \in Sh(1, p-1)} \epsilon (\sigma) f(\pi(z, X_{\sigma(1)}), \ba(X_{\sigma(2)}), \cdots, \ba(X_{\sigma(p)}))\\
&+ \sum_{k=1}^{p-1} (-1)^k \sum_{\sigma \in Sh(1, p-k-1)}\epsilon(\sigma) \sum_{s=1}^{n-1} f(\alpha(z), \ba(X_1), \cdots, \ba(X_{k-1}), \alpha(X_k^1), \cdots, \alpha(X_k^{s-1}), \\& \pi(X_k^s, X_{\sigma (k+1)}),\alpha(X_k^{s+1}), \cdots, \alpha(X_k^{n-1}), \ba(X_{\sigma (k+2)}), \cdots, \ba(X_{\sigma (p)}))\\
&=\sum_{i=1}^p (-1)^{i-1} f([z, X_i], \ba(X_1), \cdots, \hat{X_i}, \cdots, \ba(X_p))\\
&+ \sum_{k=1}^{p-1} (-1)^k \sum_{j=1}^{p-k} (-1)^{j-1} \sum_{s=1}^{n-1} f(\alpha(z), \ba(X_1), \cdots, \ba(X_{k-1}), \alpha(X_k^1), \\
&\cdots, \alpha(X_k^{s-1}), \pi(X_k^s, X_{k+j}),
\alpha(X_k^{s+1}), \cdots, \alpha(X_k^{n-1}), \ba(X_{k+1}), \cdots, \hat{X_{k+j}}\cdots \ba(X_p))\\
&= \sum(-1)^{i-1} f([z, X_i], \ba(X_1), \cdots, \hat{X_i}, \cdots,\ba(X_p))\\
&+\sum_{k=1}^{p-1}\sum_{j=1}^{p-k} (-1)^{k+j-1} f(\alpha(z), \ba(X_1), \cdots, \ba(X_{k-1}), [X_k, X_{k+j}], \ba(X_{k+1}), \cdots, \hat{X_{k+j}}, \cdots, \ba(X_p)).
\end{array}$$

On the other hand,
$$\begin{array}{ll}
& i_\pi (f) (z, X_1, \cdots, X_p)\\
&= \sum_{k=0}^1(-1)^{(p-1)k} \sum_{\sigma in Sh(p-1, 1-k)} \epsilon(\sigma) \pi \circ_k^\sigma f(z, X_1, \cdots, X_p)\\
&= \sum_{\sigma \in Sh(p-1, 1)} \epsilon(\sigma) \pi \circ_0^\sigma f(z, X_1, \cdots, X_p) + (-1)^{p-1} \sum_{\sigma \in Sh(p-1, 0)} \epsilon(\sigma) \pi \circ_1^\sigma f(z, X_1, \cdots, X_p)\\
&= \sum_{1\leq i\leq p} (-1)^{p-i}[f(z, X_1, \cdots, \hat{X_i}, \cdots, X_p), \bar{\alpha}^{p-1}X_i] \\
&+ (-1)^{p-1}\sum_{s=1}^{n-1}[\alpha^{p-1}z, \alpha^{p-1}(X_1^1), \cdots, \alpha^{p-1}(X_1^{s-1}), f(X_1^s, X_2, \cdots, X_p), \alpha^{p-1}(X_1^{s+1}), \cdots, \alpha^{p-1}(X_1^{n-1})].
\end{array}$$

Hence,
$$\begin{array}{ll}
&[f, \pi]= \sum(-1)^{i-1} f([z, X_i], \ba(X_1), \cdots, \hat{X_i}, \cdots, \ba(X_p))\\
&+\sum_{1\leq i<j}^p (-1)^{j-1} f(\alpha(z), \ba(X_1), \cdots, \ba(X_{i-1}), [X_i, X_{j}], \ba(X_{i+1}), \cdots, \hat{X_{j}}, \cdots, \ba(X_p))\\
&+ (-1)^p \sum_{1\leq i\leq p} (-1)^{p-i}[f(z, X_1, \cdots, \hat{X_i}, \cdots, X_p), \ba^{p-1}(X_i)] \\
&+(-1)^p (-1)^{p-1}\sum_{s=1}^{n-1}[\alpha^{p-1}(z), \alpha^{p-1}(X_1^1), \cdots, \alpha^{p-1}(X_1^{s-1}), f(X_1^s, X_2, \cdots, X_p), \alpha^{p-1}(X_1^{s+1}), \\
&\cdots, \alpha^{p-1}(X_1^{n-1})]\\
&= - \delta^p f.
\end{array}$$

The following lemma plays the fundamental role in showing that $\Gamma^*(L)$ with the bracket $[\cdot,\cdot ]_{\mathcal{N}}$ forms a graded Lie algebra.
\begin{lemma}\label{com}
Let $L$ be a $n$-Hom-Leibniz algebra and $\mathfrak{g}$ be the Hom-Leibniz algebra $\mathcal{D}_{n-1}(L)$. Let  $M^p(\mathfrak{g})= \mbox{Hom}(\mathfrak{g}^{\otimes p+1}, \mathfrak{g})$, and let $\Gamma^p(L)= C^{p+1}(L, L)$. Then, the following diagram commutes.
\[
\xymatrix{
M^p(\mathfrak{g}) \otimes M^q(\mathfrak{g})\ar[r]^{[.,.]_L} & M^{p+q}(\mathfrak{g}) \\
  \Gamma^p(L) \otimes \Gamma^q(L) \ar[r]^{[.,.]_{\mathcal{N}}} \ar[u]^{\Delta \otimes \Delta} &\Gamma^{p+q}(L) \ar[u]^\Delta
}
\]
\end{lemma}

{\bf Proof.} Let $f\in \Gamma^p(L)$ and $g\in \Gamma^q(L)$. We need to show that $[\Delta f, \Delta g]_L= \Delta [f, g]_\mathcal{N}.$

Now, $[\Delta f, \Delta g]_L (X_0, X_1, \cdots ,X_{p+q})=\{ i_{\Delta f}(\Delta g) + (-1)^{pq +1} i_{\Delta g} (\Delta f)\} (X_0, X_1, \cdots ,X_{p+q})$. On the other hand, 

\begin{equation}
\begin{array}{ll}
&\Delta[f, g]_\mathcal{N}(X_0, X_1, \cdots, X_{p+q}) \\
&= \sum_{i=1}^{n-1} \alpha^{p+q-1}X_0^1 \otimes \cdots \otimes \alpha^{p+q-1} X_0^{i-1}\otimes [f, g]_\mathcal{N}(X_0^i, X_1, \cdots, X_{p+q}), \alpha^{p+q-1}X_0^{i+1} \otimes \\
& \hspace{1cm}  \cdots \otimes \alpha^{p+q-1}X_0^{n-1} \\
&= \sum_{i=1}^{n-1} \alpha^{p+q-1}X_0^1 \otimes \cdots \otimes \alpha^{p+q-1} X_0^{i-1}\otimes \{i_f(g) + (-1)^{pq+1} i_g(f)\}(X_0^i, X_1, \cdots, X_{p+q})\otimes \\
&\hspace{1cm} \alpha^{p+q-1}X_0^{i+1} \otimes \cdots \otimes \alpha^{p+q-1}X_0^{n-1}.
\end{array}
\end{equation}

By definition,
$$\begin{array}{ll}
& i_f(g)(X_0^i, X_1, \cdots, X_{p+q})=
\sum_{k=0}^p (-1)^{kq} \sum_{\sigma \in Sh(q, p-k)} \epsilon(\sigma) f \circ_k^\sigma g (X_0^i, X_1, \cdots, X_{p+q}).
\end{array}$$
Now $$\begin{array}{ll}
&f \circ_k^\sigma g (X_0^i, X_1, \cdots, X_{p+q})\\
& = \sum_{s=1}^{n-1} f(\alpha^q X_0^i, \ba^q X_1, \cdots, \ba^q X_{k-1}, \alpha^q X_k^1, \cdots, \alpha^q X_k^{s-1},\\
&g(X_k^s, X_{\sigma(k+1)}, \cdots, X_{\sigma(k+q)}), \alpha^q X_k^{s+1}, \cdots, \alpha^q X_k^{n-1}, \ba^q X_{\sigma(k+q+1)}, \cdots, \ba^q X_{\sigma(p+q)}).
\end{array}$$

Sign of this term in the expression of $\Delta[f, g]_\mathcal{N}(X_0, X_1, \cdots, X_{p+q})$ is $(-1)^{kq} \epsilon(\sigma)$.

On the other hand, 
\begin{equation}\label{lhs}
\begin{array}{ll}
&i_{\Delta f}(\Delta g) (X_0, X_1, \cdots, X_{p+q})\\
&= \sum_{k=0}^p (-1)^{kq} \sum_{\sigma \in Sh(q, p-k)} \epsilon(\sigma) \Delta f \circ_k^\sigma \Delta g(X_0, X_1, \cdots, X_{p+q}).
\end{array}
\end{equation}
 
Now $$\begin{array}{ll}
&\Delta f \circ_k^\sigma \Delta g (X_0, X_1, \cdots, X_{p+q})\\
=&\Delta f (\ba^q X_0, \ba^q X_1, \cdots, \ba^q X_{k-1}, \Delta g( X_k, X_{\sigma(k+1)}, \cdots, X_{\sigma(k+q)}),\\
& \ba^q X_{\sigma(k+q+1)}, \cdots, \ba^q X_{\sigma(p+q)}).
\end{array}$$

Also,
$$\begin{array}{ll}
&\Delta g (X_k, X_{\sigma(k+1)}, \cdots, X_{\sigma(k+q)})\\
&= \sum_{i=1}^{n-1} \alpha^q X_k^1 \otimes \cdots \otimes \alpha^q X_k^{i-1} \otimes g(X_k^i, X_{\sigma(k+1)}, \cdots, X_{\sigma(k+q)})\otimes \alpha^q X_k^{i+1} \otimes \cdots \otimes \alpha^q X_k^{n-1}\\
&=A, \mbox{say}.
\end{array}$$

Replacing the above term $A$ in the expression of $\Delta f \circ_k^\sigma \Delta g (X_0, X_1, \cdots, X_{p+q})$ and writing out explicitly, we get 

$$\begin{array}{ll}
 &\Delta f \circ_k^\sigma \Delta g (X_0, X_1, \cdots, X_{p+q})\\
&=\sum_{j=1}^{n-1} \alpha^{p+q-1} X_0^1 \otimes \cdots \alpha^{p+q-1} X_0^{j-1} \otimes f(\alpha^q X_0^j, \ba^q X_1, \cdots, \ba^q X_{k-1}, A, \ba^q X_{\sigma(k+q+1)}, \cdots, \ba^q X_{\sigma(p+q)})\\
& \otimes \alpha^{p+q-1}X_0^{j+1} \otimes \cdots \otimes \alpha^{p+q-1}X_0^{n-1}.
\end{array}$$

Sign of this term is $(-1)^{kq} \epsilon(\sigma)$ in the expression \eqref{lhs} . 

Similarly, the term $\Delta g \circ_k ^\sigma \Delta f (X_0, X_1, \cdots, X_{p+q})$ is same as the term 
$$\begin{array}{ll}
&\sum_{i=1}^{n-1} \alpha^{p+q-1}X_0^1 \otimes \cdots \otimes \alpha^{p+q-1} X_0^{i-1} \otimes g\circ_k^\sigma f(X_0^i, X_1, \cdots, X_{p+q})\\
&\otimes \alpha^{p+q-1} X_0^{i+1} \otimes \cdots \otimes \alpha^{p+q-1} X_0 ^{n-1}\\
&=\sum_{i=1}^{n-1} \alpha^{p+q-1}X_0^1 \otimes \cdots \otimes \alpha^{p+q-1} X_0^{i-1} \otimes\\
& g(\alpha^p X_0^i, \ba^p X_1, \cdots, \ba^p X_{k-1}, \alpha^p X_k^1, \cdots, \alpha^p X_k^{s-1}, f(X_k^s, X_{\sigma(k+1)}, \cdots, X_{\sigma(k+p)}),\\
& \alpha^p X_k^{s+1}, \cdots, \alpha^p X_k^{n-1}, \ba^p X_{\sigma(k+p+1)}, \cdots, \ba^p X_{\sigma(p+q)})\\
&\otimes X_0^{i+1} \otimes \cdots \otimes \alpha^{p+q-1} X_0 ^{n-1}.
\end{array}$$ 

Signs of both of these terms are $(-1)^{pq+1} (-1)^{kp} \epsilon(\sigma)$. This proves the lemma.

The following proposition follows from \,Lemma \ref{com} and the fact that $\Delta$ is injective.
\begin{proposition}\label{Jacobi}
The bracket $[.,.]_{\mathcal{N}}$ so defined on $\Gamma^*(L)$ of a $n$-Hom-Leibniz algebra $L$ satisfies the graded Jacobi identity.
\end{proposition}\qed

Using Propositions \,\ref{pipi},  \,\ref{pi} and \ref{Jacobi}, we have the following proposition.

\begin{proposition} Let $f\in \Gamma^p(L), g \in \Gamma^q(L)$, then
$$\delta[f, g]_{\mathcal{N}}=(-1)^{q+1} [\delta f, g]_{\mathcal{N}} + [f, \delta g]_{\mathcal{N}}.$$
\end{proposition}
Finally, we have, 
\begin{theorem}
The induced bracket $[.,.]_{\mathcal{N}}$ of degree $-1$ defined on the cohomology space $H^*(L, L)$ of a $n$-Hom-Leibniz algebra $L$ turns it into a graded Lie algebra. \qed
\end{theorem}

\section{Deformation cohomology of $n$-Hom-Leibniz algebras}
Deformation theory was introduced first by Gerstenhaber for rings and associative algebras in \cite{Gerstenhaber} using formal power series, then extended to Lie algebras by Nijenhuis and Richardson \cite{NR}. The main results connect  the properties of  deformations to a cohomology groups of a suitable cohomology complex, which are Hochschild cohomology for associative algebras and Chevalley-Eilenberg cohomology for Lie algebras. 
 The complex governing infinitesimal deformations of a $n$-Lie algebra was successfully defined by Gautheron in \cite{Gautheron} (1996) after an initial attempt by Takhtajan \cite{Takhtajan1} (1994).
Then, Daletskii and Takhtajan \cite{Daletskii,Takhtajan} (1997) rewrote Gautheron's work in terms of a subcomplex.

The suitable cohomology in the study of $n$-Hom-Leibniz algebras deformations corresponds to  the particular case of Section 3, where $M=L$. In the sequel we provide the theory for $n$-Hom-Leibniz algebras. We may recall first, second and third coboundary maps.

Let $\mathbb{K}[[t]]$ denote the formal power series ring over $\mathbb{K}$. A formal deformation of a $n$-Hom-Leibniz algebra $(L, [\cdot,\cdot ], \alpha)$ is a map
$$f_t: L[[t]] ^{\otimes n}\longrightarrow L[[t],$$ where tensor product taken over $\mathbb{K}[[t]]$, such that for $X=(x_1, x_2,\ldots, x_n) \in L^{\otimes n}$, 
$$f_t(X)= F_0(X)+ t F_1(X) +t^2 F_2(X)+ \ldots$$  for some $F_i : L^{\otimes n}\longrightarrow L$, $i\geq 1$, $F_0$ being the bracket in $L$, and $f_t$ defining a $n$-Hom-Leibniz algebra structure on $L[[t]]$.

Explicitly, this would mean 
$$\begin{array}{ll}
&f_t(f_t(X),\alpha(y_1), \cdots,\alpha(y_{n-1}))\\
&=\sum_{i=1}^n f_t(\alpha(x_1), \cdots, \alpha(x_{i-1}), f_t(x_i, y_1, \cdots, y_{n-1}), \alpha(x_{i+1}), \cdots \alpha(x_n)).
\end{array}
$$
 
This gives rise to the following infinite system of equations:

$$\begin{array}{ll}
&\sum_{j+k=s} t^s\{ F_j(F_k(X), \alpha(y_1), \cdots, \alpha(y_{n-1}) \}\\
&= \sum_{i=1}^n \sum_{j+k=s} t^s F_j(\alpha(x_1), \ldots, \alpha(x_{i-1}), F_k(x_i, y_1, \ldots, y_{n-1}, \alpha(x_{i+1}), \ldots, \alpha(x_{n}))).
\end{array}$$

So equating coefficients of $t^s$, we get the following system of equations for all $s\geq 0$.
\begin{equation}\label{def}
\begin{array}{ll}
&\sum_{j+k=s}  F_j(F_k(X), \alpha(y_1), \cdots, \alpha(y_{n-1}))\\
& = \sum_{i=1}^n \sum_{j+k=s}  F_j(\alpha(x_1), \ldots, \alpha(x_{i-1}), F_k(x_i, y_1, \ldots, y_{n-1}), \alpha(x_{i+1}), \ldots, \alpha(x_{n})))
\end{array}
\end{equation}

For $s=0$, we get back the defining relation for the $n$-Hom-Leibniz algebra bracket.

For $s=1$, we get 
$$\begin{array}{ll}
& F_0(F_1(X), \alpha(y_1), \cdots, \alpha(y_{n-1})) + F_1 (F_0(X), \alpha(y_1), \cdots, \alpha(y_{n-1})) \\
& = \sum_{i=1}^n  \{ F_0(\alpha(x_1), \ldots, \alpha(x_{i-1}), F_1(x_i, y_1, \ldots, y_{n-1}), \alpha(x_{i+1}), \ldots, \alpha(x_{n}))\\
&+   F_1(\alpha(x_1), \ldots, \alpha(x_{i-1}), F_0(x_i, y_1, \ldots, y_{n-1}), \alpha(x_{i+1}), \ldots, \alpha(x_{n}))\}.
\end{array}$$

As $F_0$ is the bracket for $L$, we have,
\begin{equation}\label{F1}
\begin{array}{ll}
& [F_1(X), \alpha(y_1), \cdots, \alpha(y_{n-1})] + F_1 ([X], \alpha(y_1), \cdots, \alpha(y_{n-1})) \\
&= \sum_{i=1}^n  \{[\alpha(x_1), \ldots, \alpha(x_{i-1}), F_1(x_i, y_1, \ldots, y_{n-1}), \alpha(x_{i+1}), \ldots, \alpha(x_{n})]\\
&+   F_1(\alpha(x_1), \ldots, \alpha(x_{i-1}), [x_i, y_1, \ldots, y_{n-1}], \alpha(x_{i+1}), \ldots, \alpha(x_{n}))\}.
\end{array}
\end{equation}

It turns out, from equation \eqref{F1} that 
\begin{proposition} Thinking of $F_1: L^{\otimes n}\longrightarrow L$ as a map from $L\otimes L^{\otimes n-1}\longrightarrow L$, $F_1$  is a $2$-cocycle.
\end{proposition}
{\bf Proof.} Let $X=(z, x_1, \ldots, x_{n-1})$ and $Y=(y_1, \ldots, y_{n-1})$. Then equation \eqref{F1} reads as
$$\begin{array}{ll}
&[F_1(z, x_1, \ldots, x_{n-1}), \alpha(y_1), \ldots, \alpha(y_{n-1})] + F_1([z, x_1, \ldots, x_{n-1}], \alpha(y_1), \ldots, \alpha(y_{n-1}))\\
&= \sum_{i=1}^{n-1}[\alpha(z), \ldots, \alpha(x_{i-1}), F_1(x_i, y_1, \ldots, y_{n-1}), \alpha(x_{i+1}, \ldots, \alpha(x_n)] \\
&+[F_1(z, y_1, \ldots, y_{n-1}), \alpha(x_1),\ldots, \alpha(x_n)]\\
&  +\sum_{i=1}^{n-1} F_1(\alpha(z), \ldots, \alpha(x_{i-1}, [x_i, y_1, \ldots, y_{n-1}], \alpha(x_{i+1}), \ldots, \alpha(x_n))\\
&+F_1([z, y_1, \ldots, y_{n-1}], \alpha(x_1), \ldots, \alpha(x_n)).
\end{array}
$$
Rewriting, 
$$\begin{array}{ll}
&[F_1(z, X), \ba(Y)] + F_1([z, X], \ba(Y)) -  \sum_{i=1}^{n-1}[\alpha(z), \ldots, \alpha(x_{i-1}), F_1(x_i, Y), \alpha(x_{i+1}), \ldots, \alpha(x_n)] \\
&- [F_1(z, Y), \ba(X)]- F_1(\alpha(z), [X, Y]) - F_1([z, Y], \ba(X))=0.
\end{array}$$
This precisely tells us that $\delta^2(F_1)=0$ ie. $F_1$ is a $2$-cocycle.

\qed

Similarly, rewriting the system of equations \eqref{def}, we get, for each $s\geq 2$,\\

$$\delta^2( F_s)= \sum_{\stackrel{i+j=s}{ i,j >0}}i_{F_i}( F_j) = 1/2 \sum_{\stackrel{i+j=s}{ i,j>0}} [F_i, F_j]_{\mathcal{N}}.$$

\begin{definition}
Let $L_f= (L, f_t, \alpha)$ and $L_g= (L, g_t, \alpha)$ be two deformations of a $n$-Hom-Leibniz algebra $L= (L, [\cdot,..,\cdot],\alpha)$ where $f_t= \sum_{i\geq 0}t^i F_i$ and $g_t= \sum_{i\geq 0}t^i G_i$, with $F_0= G_0= [\cdot,..,\cdot]$. We say that $L_f$ and $L_g$ are equivalent if there exists a formal automorphism $\phi_t: L[[t]] \longrightarrow L[[t]]$ of the form $\phi_t= \sum_{i\geq0}\phi_i t^i$ where $\phi_i \in \mbox{End}(L)$ and $\phi_0= \mbox{id}$ such that
\begin{equation}\label{equ1}
\phi_t(f_t(x_1, \ldots, x_n))= g_t(\phi_t(x_1), \ldots, \phi_t(x_n))
\end{equation}
\begin{equation}\label{equ2}
\phi_t(\alpha(x))= \alpha(\phi_t(x)).
\end{equation}
\end{definition}

A deformation $L_f$ of $L$ is said to be trivial if and only if $L_f$ is equivalent to $L$.
\begin{remark}
Straight forward computations shows that the first term of a deformation depends only on the cohomology class, that is two cohomologous cocycles corresponds to equivalent deformations.
\end{remark}

Define a $3$-cochain $G_s= \sum_{\stackrel{i+j=s}{ i,j>0}} [F_i, F_j]_{\mathcal{N}}$ for all $s\geq 2$. This cochain is called the obstruction cochain for extending a given deformation truncated at $(s-1)$th stage to the $s$th stage.

\begin{proposition}
The obstruction cochain $G_s$ is a $3$-cocycle, and the vanishing of the cohomology class determined by $G_s$ is a necessary and sufficient condition for the deformation to be extendable.
\end{proposition}
{\bf Proof.} Let $f_t=\sum_ {i=1}^{s-1} F_i t^i$ be a truncated (at $(s-1)$th stage) deformation.    
By Proposition\,\ref{pipi}, we have,
$$\begin{array}{ll}
\delta (G_s)&=-[G_s, \pi]
=-[\sum_{\stackrel{i+j=s}{ i,j>0}} [F_i, F_j], \pi]
=-\sum_{\stackrel{i+j=s}{ i,j>0}}[[F_i, F_j], \pi]\\
& =\sum_{\stackrel{i+j=s}{ i,j>0}}\{[[F_j, \pi], F_i] +[[\pi, F_i], F_j]\} \hspace{50pt}(\mbox{by graded Jacobi identity})\\
&=2 \sum_{\stackrel{i+j=s}{ i,j>0}} [[F_i, \pi], F_j] \hspace{50pt} \mbox{(As $[F_i, \pi]= [\pi, F_i]$, both being $2$-cochains)}\\
&= 2  \sum_{\stackrel{i+j=s}{ i,j>0}} [\sum_{\stackrel{k+l=i}{k, l>0}}[F_k, F_l], F_j]\\
&=2  \sum_{\stackrel{i+j=s}{ i,j>0}} \sum_{\stackrel{k+l=i}{k, l>0}} [[F_k, F_l], F_j]\\
&= 2 \sum_{\stackrel{k+l+j=s}{k, l, j>0}} [[F_k, F_l], F_j]\\
&=0, \hspace{100pt} \mbox{(by graded Jacobi identity)}.
\end{array}
$$ It follows easily that if $G_s= \delta \phi_s$ for some $2$-cochain $\phi$, then  $f_t+ \phi t^{s}$ is a deformation truncated at $s$th stage, and is an extension of $f_t$.


\begin{thebibliography}{99}
\bibitem{Ataguema} Ataguema A., Makhlouf A., Silvestrov S., \emph{Generalization of $n$-ary Nambu algebras and beyond}, J. Math. Phys. \textbf{50}, 083501, (2009). (DOI:10.1063/1.3167801).

\bibitem{AEM} Ammar F., Ejbehi Z., Makhlouf A., \emph{Cohomology and Versal deformations of Hom-Leibniz algebras}, arXiv:1302.7302 (2013).

\bibitem{ams:MabroukTernary} Ammar F.,  Makhlouf A. and  Silvestrov S.,\emph{ Ternary $q$-Virasoro-Witt Hom-Nambu-Lie algebras,} Journal of Physics A - Mathematical and Theoretical, \textbf{43},no. 26, 265204, (2010).



\bibitem{ams:MabroukRep}  Ammar F.,  Mabrouk S.,  Makhlouf A.,\emph{ Representations and cohomology of $n$-ary multiplicative Hom-Nambu-Lie algebras}, - J. Geom.  Physics , \textbf{61}, no. 10, 1898-1913, (2011).

\bibitem{AM2008} Ataguema H.,  Makhlouf A. and Silvestrov S. \textit{
Generalization of $n$-ary Nambu algebras and beyond}, Journal of  Mathematical  Physics  \textbf{50},  no. 8, 083501, (2009).




\bibitem{BSZ} Bai R., Song G., Zhang Y., \emph{On classification of $n$-Lie algebras}, Front. Math.
China 6, 581--606 (2011).

\bibitem{BL} Bagger J. and Lambert N., \emph{Modeling multiple M2�s,} Phys. Rev. D 75 (2007). 
\bibitem{B} Balavoine, D., \emph{Deformations of algebras over a quadratic operad}, Contemporary Mathematics, \textbf{202}, (1997), 207--234. 

\bibitem{Basu} Basu A.,  Harvey J.A.,  \emph{The M2�M5 brane system and a generalized Nahm equation}, Nuclear Phys. B 713 (2005).

\bibitem{CILL} Casas J. M., Insua M. A., Ladra M., Ladra S. \emph{Test for Leibniz $n$-algebra structure,} Linear Algebra and its applications, \textbf{494},  138--155, (2016)

\bibitem{CasasLodayPirashvili}  Casas J.M., Loday J.-L. and Pirashvili  \textit{Leibniz
 $n$-algebras},
 Forum Math. \textbf{14} (2002), 189--207.

\bibitem {CS} Cheng Y. S., Su Y. C., \emph{(Co)homology and universal central extensions of Hom-Leibniz algebras}, Acta Mathematica Sinica Eng. Ser., \textbf{27}, No. 5, (2011), 813--830.

\bibitem{Daletskii}Daletskii Y.L.,  Takhtajan A.,  \emph{Leibniz and Lie structures for Nambu algebra}, Lett. Math. Phys. \textbf{39} (1997) 127--141.

\bibitem{De Azcarraga3} De Azcarraga J.~A., Izquierdo J.~M., \emph{$n$-ary algebras: a review with applications,} Journal of Physics A, Mathematical and Theoretical 07/2010; 43. DOI: 10.1088/1751-8113/43/29/293001.

\bibitem{De Azcarraga4} De Azcarraga J.~A., Izquierdo J.~M., \emph{On a class of $n$-Leibniz deformations of the simple Filippov algebras},	arXiv:1009.2709 [math-ph] (2010)



\bibitem{Fialowski86}  Fialowski A. {\it Deformation of Lie algebras}, Math USSR sbornik,
vol. 55 n 2 (1986), pp 467-473.

\bibitem{Fialowski00}  Fialowski A. and Fuchs D. , {\it Construction of Miniversal deformations
of Lie algebras }, J. Funct. Anal 161 n 1(1999), 76--110.

%


\bibitem{Fialowski-Operad} Fialowski A.; Mukherjee G., Naolekar A.,  \emph{Versal deformation theory of algebras over a quadratic operad}. Homology Homotopy Appl. \textbf{16} (2014), no. 1, 179--198.

\bibitem{Fialowski-Deform-LIe-Leibniz}Fialowski A.,  Mandal A.,  \emph{Leibniz algebra deformations of a Lie algebra,} J. Math. Phys.\textbf{ 49} (2008), no. 9, 093511, 11 pp.




\bibitem{Filippov} Filippov V.~T.~, \textit{$n$-ary Lie algebras}, Sibirskii
Math. J. \textbf{24} (1985), 126-140 (Russian).

\bibitem{Gautheron1}  Gautheron P., \emph{Some Remarks Concerning Nambu Mechanics},
Letters in Mathematical Physics \textbf{37}: 103--116, 1996.

\bibitem{Gautheron}  Gautheron P., \emph{Simple facts concerning Nambu algebras}

\bibitem{Gerstenhaber}  Gerstenhaber M. {\it On the deformations of rings
and algebras, }Ann. of Math {\bf{79}} (1964).

\bibitem{HLS}
Hartwig J. T., Larsson D.,  Silvestrov S., \emph{Deformations of Lie algebras using $\sigma$-derivations}, J.
Algebra \textbf{295} (2006), 314--361.

\bibitem{Hu}  Hu, N.,: \emph{$q$-Witt algebras,
$q$-Lie algebras, $q$-holomorph structure and
representations,}  Algebra Colloq. {\bf 6} ,
no. 1, 51--70 (1999).
\bibitem{Kasymov1} Kasymov, Sh. M., \emph{On a theory of $n$-Lie algebras}. Algebra and Logic, \textbf{26}, 155--166 (1987)

\bibitem{LarssonTA} Larsson T. A.,
\textit{Virasoro 3-algebra from scalar densities}, arXiv:0806.4039 (2008).

\bibitem{MS} Makhlouf, A. and Silvestrov, S., {\it Notes on $1$-parameter formal deformations of Hom-associative and Hom-Lie algebras} Forum Math \textbf{22}, Num. 4, (2010) 715--739.
\bibitem{Nambu} Nambu Y., \textit{Generalized Hamiltonian mechanics}, Phys.
Rev. \textbf{D7} (1973), 2405--2412.


\bibitem{NR} Nijenhuis, A. and Richardson J.R. {\it Cohomology and deformations in
graded Lie algebras} Bull. Amer. Math. Soc \textbf{72} (1966), 1--29 


\bibitem{R}Rotkiewics M., 
Cohomology Ring of $n$-Lie Algebras, extracta mathematicae Vol. \textbf{20}, Num. 3, 219 --232 (2005)

\bibitem{schliss1} Schlessinger M. {\it Functors of Artin rings, } Trans. Amer. Math. Soc.
\textbf{130} (1968),  208--222.


\bibitem{Takhtajan} Takhtajan L., \emph{On foundation of the generalized Nambu
mechanics}, Comm. Math. Phys. \textbf{160 }(1994), 295--315.

\bibitem{Takhtajan1} Takhtajan L., \emph{A higher order analog of
Chevalley-Eilenberg complex and deformation theory of $n$-algebras},
St. Petersburg Math. J. \textbf{6} (1995), 429--438.

\bibitem{Takhtajan2} Takhtajan L., \emph{Leibniz and Lie algebra structures for Nambu algebra,} Lett. Math. Phys. \textbf{39} (1997) 127--141.


\end{thebibliography}
\end{document}